\pdfoutput=1   
\documentclass[a4paper,UKenglish,cleveref, autoref, thm-restate]{lipics-v2021}
\usepackage{ifthen}
\newboolean{longversion}
\newcommand{\onlylong}[1]{\ifthenelse{\boolean{longversion}}{#1}{}}
\newcommand{\onlyshort}[1]{\ifthenelse{\boolean{longversion}}{}{#1}}

\setboolean{longversion}{false}

\hideLIPIcs  

\usepackage{tikz}
\usepackage{nicefrac}

\title{Randomness versus Superspeedability}

\author{Rupert H\"olzl}{Universit\"at der Bundeswehr M\"unchen, Germany \and \url{https://hoelzl.fr}}{r@hoelzl.fr}{}{}

\author{Philip Janicki}{Universit\"at der Bundeswehr M\"unchen, Germany \and \url{}}{philip.janicki@unibw.de}{https://orcid.org/0009-0008-9063-028X}{}

\author{Wolfgang Merkle}{Universit\"at Heidelberg, Germany \and \url{https://web.ifi.uni-heidelberg.de/tcs/merkle}}{merkle@math.uni-heidelberg.de}{}{}

\author{Frank Stephan}{National University of Singapore, Republic of Singapore \and \url{https://www.comp.nus.edu.sg/~fstephan/}}{fstephan@comp.nus.edu.sg}{}{}

\authorrunning{R. H\"olzl, P. Janicki, W. Merkle and F. Stephan} 

\Copyright{Rupert H\"olzl, Philip Janicki, Wolfgang Merkle and Frank Stephan} 

\ccsdesc[500]{Theory of computation~Computability}

\keywords{superspeedable numbers, speedable numbers, regainingly approximable numbers, regular numbers, left-computable numbers}

\category{} 

\relatedversion{} 

\funding{The investigators acknowledge the following partial support: F.~Stephan's research was supported by Singapore Ministry of Education AcRF Tier 2 grant MOE-000538-00 and AcRF Tier 1 grants A-0008454-00-00 and A-0008494-00-00, the grant A-0008494-00-00 also supported the research visits of R.~H\"olzl to the National University of Singapore for the work on this paper.}

\acknowledgements{The authors would like to thank the referees for helpful and detailed comments.}

\nolinenumbers 

\def\ut{{\upharpoonright}}
\def\IN{{\mathbb N}}

\def\IQ{{\mathbb Q}}
\def\IR{{\mathbb R}}

\def\SigmaS{\Sigma^{\ast}} 

\theoremstyle{claimstyle}

\EventEditors{Rastislav Kr\'{a}lovi\v{c} and Anton\'{i}n Ku\v{c}era}
\EventNoEds{2}
\EventLongTitle{49th International Symposium on Mathematical Foundations of Computer Science (MFCS 2024)}
\EventShortTitle{MFCS 2024}
\EventAcronym{MFCS}
\EventYear{2024}
\EventDate{August 26--30, 2024}
\EventLocation{Bratislava, Slovakia}
\EventLogo{}
\SeriesVolume{306}
\ArticleNo{54}
\begin{document}

\maketitle

\begin{abstract}
Speedable numbers are real numbers which are algorithmically approximable from below and whose approximations can be accelerated nonuniformly. We begin this article by answering a question of Barmpalias by separating a strict subclass that we will refer to as {\em superspeedable} from the speedable numbers; for elements of this subclass, acceleration is possible  uniformly  and to an even higher degree. This new type of benign left-approximation of numbers then integrates itself into a hierarchy of other such notions studied in a growing body of recent work. We add a new perspective to this study by juxtaposing this hierachy with the well-studied hierachy of algorithmic randomness notions.
\end{abstract}

\section{Introduction}

In theoretical computer science, the concepts of randomness and of computation speed, as well as their interactions with each other, play an important role. For example, one can study whether access to sources of random information can enable or simplify the computation of certain mathematical objects, for example by speeding up computations of difficult problems. This is a recurring paradigm of theoretical computer science, and there is more than one way of formalizing it, but the perhaps most well-known instance is the difficult open question in the field of complexity theory of whether $\mathsf{P}$ is equal to $\mathsf{BPP}$.

A recent line of research could be seen as investigating the reverse direction: to what degree can objects that are random be quickly approximated? Before answering this question, we need to define what we understand by random. The answer is provided by the research field of information theory that deals with questions of compressibility of information, where randomly generated information will be difficult to compress, due to its lack of internal regularities. The subfield of algorithmic randomness, which set out to unify these and other topics and to understand what actually makes a mathematical object random, has produced a wealth of insights into questions about computational properties that random objects must, may, or may not possess. The general pattern in this area is that we put at our disposal algorithmic tools to detect regularities in sequences. Those sequences that are ``random enough'' to resist these detection mechanisms are then considered random. Many different ``resistance levels'' have been identified that give rise to randomness notions of different strengths.

For completeness, we define the two central randomness notions featuring in this article,  introduced by Martin-L\"of~\cite{MARTINLOF1966602} and Schnorr~\cite{Sch71} respectively: A {\em randomness test} is a sequence  $(U_n)_{n\in\IN}$ of uniformly effectively open subsets of $\IR$ and we say that it {\em covers} an $x \in \IR$ if and only if $x \in \bigcap_n U_n$. 
Write $\mu$ for the Lebesgue measure on $\IR$; if $\mu(U_n) \leq 2^{-n}$ for all~$n\in \IN$, then we call $(U_n)_{n\in\IN}$ a {\em Martin-Löf test}, and if an $x \in \IR$ is not covered by any such test, then we call it {\em Martin-L\"of random}.
If we even require that $\mu(U_n) = 2^{-n}$ for all $n\in \IN$, then we call $(U_n)_{n\in\IN}$ a {\em Schnorr test}, and an $x \in \IR$ that is not covered by any such test {\em Schnorr random}. It is well-known that Schnorr randomness is a strictly weaker notion than Martin-Löf randomness; that is, there exist more Schnorr random numbers than Martin-Löf random ones. For the rest of the article, we assume that the reader is familiar with these topics; see Calude~\cite{Ca13}, Downey and Hirschfeldt~\cite{DH2010}, Li and Vit\'anyi~\cite{LV08}, or Nies~\cite{Nie09} for comprehensive overviews.

With those randomness paradigms in place, we can study the interactions between randomness and computation speeds. We will focus on a class of objects that can be approximated by an algorithm, and will be interested in the possible speeds of these approximations.
\begin{definition}
Let $x$ be a real number. For a fixed sequence $(x_n)_n$ of real numbers converging to it from below, write $\rho(n)$ for the quantity
$\frac{x_{n+1} - x_n}{x - x_n}$
for each~$n \in \IN$, and call $(\rho(n))_n$ the {\em speed quotients} of~$(x_n)_n$.
\end{definition}
\begin{definition}\label{def:speedable}
	If there exists a {\em left-approximation} of~$x$, that is, a computable increasing sequence of rational numbers converging to~$x$, then we call $x$ {\em left-computable}. If there even exists a left-approximation $(x_n)_n$ with speed quotients $(\rho(n))_n$ such that $\rho:=\limsup_{n\to\infty} \rho(n) > 0$, then we call $x$ \emph{speedable}.
\end{definition}
Left-computable numbers can be Martin-Löf random; the standard example is {\em Chaitin's $\Omega$}, the measure of the domain of an optimally universal prefix-free Turing machine. It is also known that 
there exist left-computable  Schnorr randoms that are not Martin-L\"of random.

The notion of speedability was introduced by Merkle and Titov. They made the interesting observation that speedability is incompatible with Martin-L\"of randomness~\cite[Theorem~10]{MT20} and asked the question whether the converse holds, that is, whether there exist left-computable numbers which are neither speedable nor Martin-L\"of random. A positive answer would have been an interesting new characterization of Martin-L\"of randomness; however, H\"olzl and Janicki~\cite[Corollary 62]{HJ2023b} showed that such a characterization does {\em not} hold. They also introduced the term {\em benign approximations} when referring to subclasses of the left-computable numbers that possess special approximations that are in some sense better behaved than general approximations from below. One example is speedability; another that will play a role in this article is the following.
\begin{definition}[Hertling, H\"olzl, Janicki~\cite{HHJ2023}]
	A real number $x$ is called \emph{regainingly approximable} if there exists a computable increasing sequence of rational numbers $(x_n)_n$ converging to $x$ with $x - x_n \leq 2^{-n}$ for infinitely many $n \in \IN$.
\end{definition}
Obviously, every regainingly approximable number is left-computable and every computable number is regainingly approximable, but none of these implications can be reversed by results of Hertling, H\"olzl, Janicki~\cite{HHJ2023}.

Merkle and Titov also showed that the property of being speedable does not depend on the choice of $\rho$, that is, if some left-computable number is speedable via some $\rho$ witnessed by some left-approximation, it is also speedable via any other $\rho' < 1$ witnessed by some other left-approximation. However, their proof is highly nonuniform, which led to Barmpalias~\cite{BarmDirect} asking whether every such number also has a {\em single} left-approximation with $\limsup_{n\to\infty} \rho(n) = 1$. 

We answer this question negatively in this article by showing that the requirement to achieve $\limsup_{n\to\infty} \rho(n) = 1$ is a strictly stronger condition that defines a strict subset of the speedable numbers which we will call superspeedable.
We then show that the regainingly approximable numbers are strictly contained in this subset. Thus within the speedable numbers there is a strict hierarchy of notions of benign approximations, namely computable implies regainingly approximable implies superspeedable implies speedable.

Wu~\cite{Wu05a} called finite sums of binary expansions of computably enumerable sets {\em regular reals}. We show that these reals are a strict subset of the superspeedable numbers as well; in particular every binary expansion of a computably enumerable set is superspeedable. 
Then we show that finite sums of left-computable numbers can only be superspeedable if at least one of their summands is superspeedable.

With this new approximation notion identified and situated in the hierarchy of benign approximability, it is then natural to wonder how it interacts with randomness. We will show that unlike the Martin-L\"of randoms, which by the result of Merkle and Titov~\cite{MT20} cannot even be speedable, the Schnorr randoms may even be superspeedable; however they cannot be regainingly approximable.
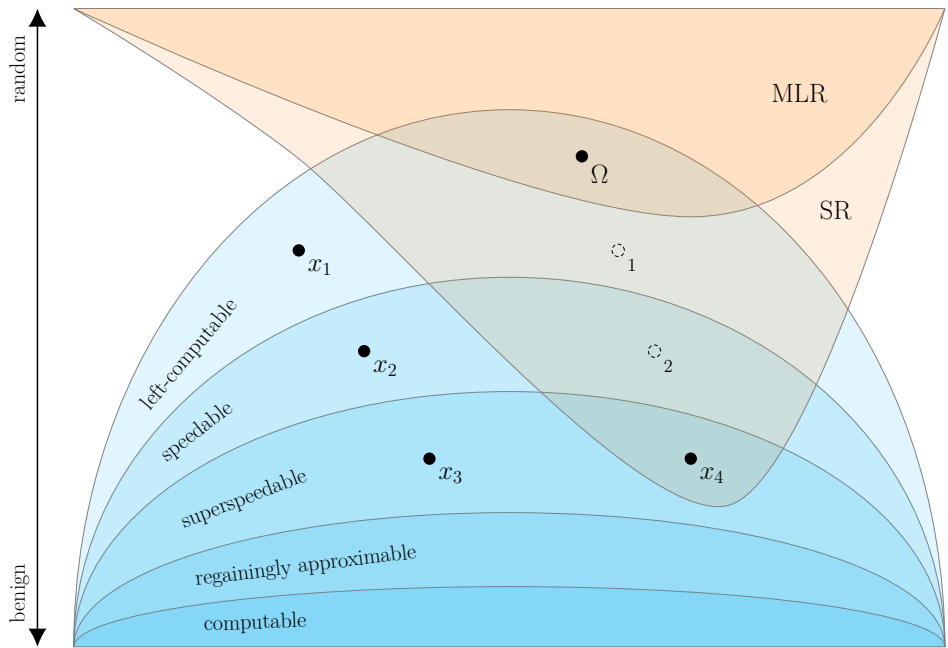
\begin{figure}[t]
	\begin{center}
		\scalebox{0.48}{
		\usetikzlibrary{arrows.meta}
		\begin{tikzpicture}[
			yscale=1.85,xscale=2,
			set/.style={fill=cyan,fill opacity=0.125, draw=gray},
			orangeset/.style={fill=orange,fill opacity=0.125, draw=gray},
			point/.style = {circle, fill=black, inner sep=0.12cm,	node contents={}},
			unclear/.style = {draw=black,thick,circle, densely dashed, inner sep=0.12cm,node contents={}}
			]
			
			\path  (-7,-0.5) rectangle (7,-10); 

			\draw[set]  (6,-10) node (v2) {} arc(0:180:6 and 5.5);
			\draw[set]  (v2) node (v3) {} arc(0:180:6 and 3.8);
			\draw[set]  (v3) node (v4) {} arc(0:180:6 and 2);
			\draw[set]  (v4) node (v5) {} arc(0:180:6 and 0.9);
			\draw[set]  (v5) arc(0:180:6 and 8);

			\draw[gray] (-6,-0.5) node (v6) {} -- (6,-0.5) node (v7) {};
			\draw[orangeset]   plot[smooth, tension=.7] coordinates {(v6) (2.5,-3.6) (v7)} (v6);
			\draw[orangeset]   plot[smooth, tension=0.5] coordinates {(v6)  (-3,-2.65)  (3,-7.9)   (v7)} (v6);
			
			\draw[gray] (-6,-10) node (v1) {} -- (6,-10);

			\node[align=center] at (4,-1.75) {\huge $\mathrm{MLR}$};
			\node[align=center] at (4.5,-3.5) {\huge $\mathrm{SR}$};
			\node[align=center,rotate=50] at (-4.4,-5.65) {\LARGE left-computable};
			\node[align=center,rotate=38] at (-4.3,-6.8) {\LARGE speedable};
			\node[align=center,rotate=21] at (-3.65,-7.8) {\LARGE superspeedable};
			\node[align=center,rotate=7] at (-2.8,-8.8) {\LARGE regainingly approximable};
			\node[align=center,rotate=2] at (-3.5,-9.65) {\LARGE computable};

			\node at (-2.9,-4.1) [point, label={[label distance=-0cm]below right:\huge $x_1$}];
			\node at (-2,-5.6) [point, label={[label distance=-0cm]below right:\huge $x_2$}];
			\node at (-1.1,-7.2) [point, label={[label distance=-0cm]below right:\huge $x_3$}];
			\node at (2.5,-7.2) [point, label={[label distance=-0cm]below right:\huge $x_4$}];
			\node at (1,-2.7) [point, label={[label distance=-0cm]below right:\huge $\Omega$}];
			\node at (1.5,-4.1) [unclear, label={[label distance=-0cm]below right:\huge $_1$}];
			\node at (2,-5.6) [unclear, label={[label distance=-0cm]below right:\huge $_2$}];

			\draw[black,very thick,{Latex[length=5mm, width=5mm]}-{Latex[length=5mm, width=5mm]}] (-6.5,-0.5) node (v6) {} -- (-6.5,-10) node (v7) {};
			\node[align=center,rotate=90] at (-6.75,-9.2) {\LARGE benign};
			\node[align=center,rotate=90] at (-6.75,-1.4) {\LARGE random\vphantom{g}};

		\end{tikzpicture}
	}
	\end{center}
	\caption[Dummy LoF entry]{The number $x_1$ was constructed by H\"olzl and Janicki~\cite[Corollary 62]{HJ2023b}; by construction it is nearly computable, thus weakly $1$-generic by a result of Hoyrup~\cite{Hoy17}, thus not Schnorr random (see, for instance, Downey and Hirschfeldt~\cite[Proposition 8.11.9]{DH2010}). The number $x_2$ constructed in Theorem~\ref{satz-non-immune-non-superspeedable} cannot be Schnorr random either, because it is by construction not immune. To see that $x_3$~exists, consider the example used in the proof of Corollary~\ref{dsfdasjbndsnvnmasdfs}; it is strongly left-computable, thus not immune, thus not Schnorr random. Finally, $x_4$ is constructed in Theorem~\ref{fkhejkvfjksjkbasdsgad}.
	
	It is an open question whether numbers $\scalebox{0.55}{\begin{tikzpicture}[unclear/.style = {draw=black,thick,circle, densely dashed, inner sep=0.12cm,node contents={}}]
			\node at (0,0) [unclear];
	\end{tikzpicture}}_{\,1}$ or $\scalebox{0.55}{\begin{tikzpicture}[unclear/.style = {draw=black,thick,circle, densely dashed, inner sep=0.12cm,node contents={}}]
			\node at (0,0) [unclear];
	\end{tikzpicture}}_{\,2}$ exist; see the discussion in Section~\ref{finalsection}.
 }\label{sdfjsdfjdfmndsnmvbmnsdjkbf}
\end{figure}

\section{Speedability versus immunity}

Merkle and Titov~\cite{MT20} showed that all non-high and all strongly left-computable numbers are speedable. On the other hand they established that  Martin-L\"of randoms cannot be speedable. In this section we show that a left-computable number which is not immune has to be speedable.

We first observe that, as an alternative to Definition~\ref{def:speedable} above, speedability could also be defined using two other forms of speed quotients; this will prove useful in the following.
\begin{samepage}
	\begin{proposition}\label{characterisierung-speedable}
		Let $x \in \IR$. For an increasing sequence of rational numbers $(x_n)_n$ converging to $x$ with speed quotients $(\rho(n))_n$ the following statements are easily seen to be equivalent:
		\begin{enumerate}
			\item $\limsup_{n\to\infty} \rho(n) > 0$, that is, $(x_n)_n$ witnesses the speedability of $x$;
			\item $\limsup_{n\to\infty} \frac{x_{n+1} - x_n}{x - x_{n+1}} > 0$;
			\item $\liminf_{n\to\infty} \frac{x - x_{n+1}}{x - x_n} < 1$.\qed
		\end{enumerate}
	\end{proposition}
\end{samepage}

For any set $A \subseteq \IN$, we define the real number $x_A := \sum_{n \in A} 2^{-(n+1)}\in\left[0, 1\right]$. Clearly, $A$~is computable if and only if $x_A$ is computable. If $A$ is only assumed to be computably enumerable, then $x_A$ is a left-computable number; the converse is not true as pointed out by Jockusch (see Soare~\cite{Soa69a}). If there does exist a computably enumerable set $A \subseteq \IN$ with $x_A = x$, then $x \in \left[0,1\right]$ is called \emph{strongly left-computable}.

Recall that an infinite set $A \subseteq \IN$ is \emph{immune} if it does not have an infinite computably enumerable (or, equivalently, computable) subset. We call it {\em biimmune} if both $A$~and its complement~$\bar{A}$ are immune. A real number $x \in \left[0,1\right]$ is called \emph{immune} if there exists an immune set $A \subseteq \IN$ with~$x_A = x$; analogously for {\em biimmune}.%
\begin{theorem}\label{satz:non-immune-implies-speedable}
	Every left-computable number that is not immune is speedable.
\end{theorem}
Before we give the proof we mention that left-computable numbers 
\begin{itemize}
	\item whose binary expansion is not biimmune or
	\item whose binary expansion is of the form $A \oplus A$
\end{itemize}
can be shown to be speedable by similar arguments as below.
\begin{proof}[Proof of Theorem~\ref{satz:non-immune-implies-speedable}]	
Let~$x$ be a left-computable real that is not immune. If it is computable, the theorem holds trivially, thus assume otherwise. Let~$A$ be such that~$x = x_A$; then there exists a computable increasing function $h \colon \IN \to \IN$ with $h(\IN) \subseteq A$. Let $(B[n])_n$ be a computable sequence of finite sets of natural numbers such that the sequence $(x_{B[n]})_n$ is increasing and converges to~$x_A$. Define the function $r \colon \IN\to \IN$ recursively via
		\[r(0) := 0, \quad
		r(n+1) := \min\{m > r(n) \colon\; \{h(0), \dots, h(n) \} \subseteq B[m] \}.\]
It is easy to see that $r$ is well-defined, computable and increasing.
	Let 
	\begin{equation*}
		s_n := \min\{m \in \IN \colon\;  m \in B[r(n+1)] \setminus B[r(n)] \}
	\end{equation*}
	for all $n \in \IN$. Clearly, the sequence $(s_n)_n$ is well-defined and tends to infinity. Finally, define the sequence $(C[n])_n$ via
		\[C[0] := \emptyset, \quad C[n+1] := B[r(n+1)] \restriction (s_n + 1),\]
	for all $n\in\IN$. Clearly, $(C[n])_n$ is a computable sequence of finite sets of natural numbers, which converge pointwise to~$A$, and the sequence $(x_{C[n]})_n$ is increasing and converges to $x$. 

We claim that there are infinitely many~$n$ such that~$s_n < h(n)$; this is because otherwise, for almost all~$n$, the sets~$B[r(n)], B[r(n+1)], \dots$ would all mutually agree on their first $h(n)-1$ bits, hence their limit~$A$ would be computable, contrary to our assumption. Thus, fix one of the infinitely many~$n$ with $s_n < h(n)$ and let~$n'\ge n$ be maximal such that~$s_{n'} < h(n)$. Then the sets~$B[r(n'+1)], B[r(n'+2)], \dots$ all contain~$h(n)$ and, by definition of~$s$ and~$n'$, agree mutually on their first~$(h(n)-1)$ bits; in other words, they even agree on their first~$h(n)$ bits.  But by choice of~$n'$, these sets also agree with the sets~$C[n'+1], C[n'+2], \dots$, respectively, on their first~$h(n)$ bits; thus these latter sets must also contain~$h(n)$ and agree mutually on their first~$h(n)$ bits. By definition, the set~$C[n']$ does \textit{not} contain~$h(n)$; thus, in summary, we have
$x_{C[n'+1]} - x_{C[n']} \geq 2^{-h(n)}$ and $x - x_{C[n'+1]} \leq 2^{-h(n)}$, hence
$\frac{x_{C[n'+1]} - x_{C[n']}}{x - x_{C[n'+1]}} \geq 1$.

Since there are infinitely many such~$n$ and corresponding~$n'$, the left-approximation~$(x_{C[n]})_n$ of~$x$ witnesses that~$x$ is speedable.
\end{proof}

\section{Superspeedability}

We now strengthen the conditions in Proposition~\ref{characterisierung-speedable} to obtain the following new definition.
\begin{definition}\label{characterisierung-superspeedable}
	Let $x \in \IR$. For an increasing sequence $(x_n)_n$ converging to $x$ with speed quotients $(\rho(n))_n$ the following statements are easily seen to be equivalent: 
	\begin{enumerate}
		\item $\limsup_{n\to\infty} \rho(n) = 1$;

		\item $\limsup_{n\to\infty} \frac{x_{n+1} - x_n}{x - x_{n+1}} = \infty$;
		
		\item $\liminf_{n\to\infty} \frac{x - x_{n+1}}{x - x_n} = 0$.
		
	\end{enumerate}
We call $x$ \emph{superspeedable} if it has a left-approximation~$(x_n)_n$ with these properties.
\end{definition}
The following statement, which bears some similarity to Theorem~\ref{satz:non-immune-implies-speedable}, shows that one simple cause for a left-computable number to be superspeedable is for it to have a left-approximation which permanently leaves arbitrarily long blocks of bits fixed to~$0$.
\begin{proposition}\label{asdbsdfjhsadnjdfgnassdfasdfgh}
Fix a sequence $(I_i)_{i\in \IN}$ of disjoint finite intervals in $\IN$ with the property that $|I_i| \geq i$ for all $n$. If $x\in \IR$ has a left-approximation $(x_n)_{n\in \IN}$ consisting only of dyadic rationals and such that for all $i\in\IN$ and $n\in \IN$ we have $x_n(j)=0$ for all $j\in I_i$, then $x$ is superspeedable.
\end{proposition}
\begin{proof}
	Fix any~$i$ and let~$n$ be maximal such that $x_n(j)\neq x(j)$ for some $j<\min I_i$; that is, let $n+1$ be the last stage at which a bit left of $I_i$ changes during the approximation of~$x$. Then, on the one hand, $x_{n+1}-x_n \geq 2^{-\min(I_i)}$. On the other hand, by the maximality of~$n$ and by the assumptions on $(x_n)_{n\in \IN}$, we must have $x-x_{n+1} \leq 2^{-\max(I_i)}$. Therefore we have
	\[\frac{x-x_{n+1}}{x-x_n} \leq \frac{x-x_{n+1}}{x_{n+1}-x_n} \leq \frac{2^{-\max(I_i)}}{2^{-\min(I_i)}} \leq  2^{-(i-1)},\]
	which tends to~$0$ as $i$ tends to infinity.
\end{proof}

Obviously, every superspeedable number is speedable.
Barmpalias~\cite{BarmDirect} asked whether the converse is true as well. We give a negative answer with the following theorem and corollary.
\begin{theorem}\label{satz-non-immune-non-superspeedable}
There is a left-computable number that is neither immune nor superspeedable.
\end{theorem} 
\begin{proof}
	Let $r\colon\IN \to \IN$ be the function defined by $r(n):=\max\{\ell \in \IN \colon\;  2^\ell \text{ divides } n+1\}$, that is, the characteristic sequence of~$r$ is the member of Baire space with the initial segment
	\[010201030102010401020103\dots\]
We construct a number $\alpha$ by approximating its binary expansion $A$. We will slightly abuse notation and silently identify real numbers with their infinite binary expansions; similarly we identify finite binary sequences $\sigma$ with the rational numbers $0.\sigma$.

We approximate~$A$ in stages where during stage~$s\in\IN$ we define a set~$A[s]$ such that the sets~$A[s]$ converge pointwise to~$A$. For all~$s$ and all $n$ with $r(n)=0$, let~$A[s](n)=0$; and for all~$s$ and all $n$ with $r(n)=1$, let~$A[s](n)=1$. We refer to all remaining bits as the {\em coding bits}.
More precisely, we refer to all bits with $r(n)=e+2$ as the {\em $e$-coding bits} and we let $(c^e_i)_i$ be the sequence of the positions of all $e$-coding bits in ascending order.
	
If $b=c^e_i$ for some $e$ and $i$, then we say that position~$b$ is {\em threatened at stage $s$} if $A[s](b)=0$ and $A[s]\ut c^e_{i+1} = B_e[s] \ut c^e_{i+1}$, where~$(B_0[s])_s, (B_1[s])_s, \ldots$ is some fixed uniformly effective enumeration of all left-computable approximations.	

\begin{claimproof}[Construction]
		At every stage~$s$, if there exists a least coding bit $b \leq s$ that is threatened, then set $A[s+1](b)=1$ and $A[s+1](c)=0$ for all coding bits $c>b$. For all other $n\in\IN$, leave bit~$n$  unchanged, that is, $A[s+1](n)=A[s](n)$.\phantom{\qedhere}
\end{claimproof}

\begin{claimproof}[Verification] 
By construction, once a coding bit has been set to $1$, it can only be set back to~$0$ if a coding bit at an earlier position is set from $0$ to $1$. From this it follows by an obvious inductive argument that $A[s]$ converges to some limit $A$ in a left-computable fashion; that is, interpreted as real numbers, the values~$A[0], A[1], \ldots$ are nondecreasing. Thus~$\alpha$ is left-computable.\phantom{\qedhere}
	
As it contains all numbers~$n$ with $r(n)=1$, the set $A$ we are constructing is trivially not immune. All that remains to show is that it is not superspeedable.  To that end, fix any~$e$ and assume that that the left-computable approximation~$(B_e[s])_s$ converges to~$A$; as otherwise there is nothing to prove for~$e$. We will show that this left-computable approximation  does not witness superspeedability of~$A$. 
\end{claimproof} 
	
\begin{claim*}	
There are stages~$s_0 < s_1 < \ldots$ such that for all~$i\ge 0$, at stage~$s_i$,  
the coding bit~$b_i=c^e_{i}$ is threatened for the last time and is the least coding bit that is threatened at this stage, and such that $A$ has already converged at this stage up to its~$(b_i+4)$-th bit; that is, for all~$\ell \geq 1$ we have~$A[s_i] \ut (b_i+4) = A[s_i+\ell] \ut (b_i+4) = A \ut (b_i+4)$.
\end{claim*}	
\begin{claimproof}
Fix some $e$-coding bit~$b_i=c^e_{i}$. By the discussion preceding the claim, there is some least stage~$s$ such that set~$A$ has already converged up to its~$b_i$-th bit. By construction, it must hold that~$A[s](b_i)=0$ and
since the left-approximation~$(B_e[s])_s$ converges to~$A$, there must be a stage $s'>s$ at which $b_i$~becomes threatened and is the least threatened coding bit. Thus, during stage~$s'$, the bit~$A(b_i)$ is set permanently to~$1$ and, because the two bits following~$b_i$ are no coding bits, $A$~has now converged up to its~$(b_i+4)$-th bit.  
Therefore, it suffices to set~$s_i=s'$ and to observe that we have~$s_i < s_{i+1}$ for all~$i$ because $e$-coding bit~$c^e_{i+1}$ is set to~$0$ at stage~$s_i$, hence is threatened again after stage~$s_i$.
\end{claimproof}
Let~$s_0, s_1, \ldots$ be the stages from the last claim.
\begin{claim*}
There is an $\rho_e<1$ only depending on $e$ such that for all~$i$ and~${s \in \{s_i,  \ldots, s_{i+1}-1\}}$,
\[
\frac{B_e[s+1]-B_e[s]}{A-B_e[s]} \le \rho_e .
\]
\end{claim*}
\begin{claimproof}
Fix some~$i$ and some stage~$s$ as above. Let~$b=c^e_{i}$ and let~$b'=c^e_{i+1}$. Then~$A$ has already converged up to its~$b$-th bit at stage~$s_i$ and by definition of a threat we must have~${B_e[s_i]\ut b =A[s_i]\ut b}$. But the left-approximation~$(B_e[s])_s$ converges to~$A$, hence can never ``overshoot''~$A$ and thus at stage~$s_i$ has already converged up to its~$b$-th bit as well. Consequently, we have~$B_e[s+1]-B_e[s]\le 2^{-(b-1)}$. 

By construction and our assumptions,
\begin{itemize}
	\item at stage~$s_{i+1}$, bit~$b'$ is set from~$0$ to~$1$,
	\item bit~$b'+1$ permanently maintains value~$0$ and,
	\item by definition of a threat, $A[s_{i+1}]$~and~$B_e[s_{i+1}]$ agree on their first~$b'+2 < c^e_{i+2}$ many bits.
\end{itemize}
This implies that
\[
A-B_e[s+1] 
\ge A[s_{i+1}+1] - B_e[s_{i+1}] 
\ge 2^{-(b'+1)} = 2^{-(b + 2^{e+3} + 1)}, 
\]
and in summary we have
\[
\begin{split}
\frac{B_e[s+1] - B_e[s]}{A - B_e[s]} 
&= \frac{B_e[s+1] - B_e[s]}{(A - B_e[s+1])+(B_e[s+1] - B_e[s])} 
\\
&
\leq \frac{2^{-(b-1)}}{2^{-(b + 2^{e+3} + 1)}+2^{-(b-1)}}
\leq \frac{ 2^{2+2^{e+3}}}{1 + 2^{2+2^{e+3}}}<1,
\end{split}
\]	
hence there is a constant~$\rho_e<1$ as claimed.
\end{claimproof}
Since~$e$ was arbitrary such that~$(B_e[s])_s$ converges to~$A$, there cannot be a left-approximation that witnesses superspeedability of $A$.
\end{proof}
By Theorem~\ref{satz:non-immune-implies-speedable}, the number $\alpha$ constructed in Theorem \ref{satz-non-immune-non-superspeedable} is speedable. 
\begin{corollary}
There exists a speedable number which is not superspeedable.\qed
\end{corollary}
Thus, superspeedability is strictly stronger than speedability; we now proceed to identify interesting classes of numbers that have this property. We begin with the class of regainingly approximable numbers. It was shown by H\"olzl and Janicki~\cite{HJ2023b} that regainingly approximable numbers are speedable. This result can be strengthened; to show that regainingly approximable numbers are in fact superspeedable, we will use the following characterization.

\goodbreak

\begin{proposition}[Hertling, H\"olzl, Janicki \cite{HHJ2023}]\label{prop:characterization-reg-app}
For a left-computable number~$x$ the following are equivalent:
\begin{enumerate}[(1)]
\item $x$ is regainingly approximable.
\item For every computable unbounded function $f \colon \IN \to \IN$ there exists a computable increasing sequence of rational numbers $(x_n)_n$ converging to $x$ with $x - x_n \leq 2^{-f(n)}$ for infinitely many $n \in \IN$.
\end{enumerate}
\end{proposition}

\begin{theorem}\label{nsdfjsfjkldasjfnadsjsfdnsdadnjsad}
Every regainingly approximable number is superspeedable.
\end{theorem}

\begin{proof}
Let $x$ be a regainingly approximable number, and fix a computable increasing sequence $(x_n)_n$ of rational numbers that converges to $x$ and satisfies $x - x_n \leq \frac{1}{(n+1)!}$ for infinitely many $n \in \IN$. Define the sequence $(y_n)_n$ with $y_n := x_n - \frac{1}{n!}$ for all $n \in \IN$. Surely, this sequence is computable, increasing and converges to $x$ as well. Considering any of the~$n \in \IN$ with $x - x_n \leq \frac{1}{(n+1)!}$, we obtain
	\begin{equation*}
		\frac{y_{n+1} - y_n}{x - y_n} 
		= \frac{(x_{n+1} - x_n) + \frac{1}{n!} - \frac{1}{(n+1)!}}{(x - x_n) + \frac{1}{n!}} 
		> \frac{\frac{1}{n!} - \frac{1}{(n+1)!}}{\frac{1}{n!} + \frac{1}{(n+1)!}}
		= \frac{\frac{n}{(n+1)!}}{\frac{n+2}{(n+1)!}} = \frac{n}{n+2}.
	\end{equation*}
	Thus we have $\limsup_{n\to\infty} \frac{y_{n+1} - y_n}{x - y_n} = 1$. Therefore, $x$ is superspeedable.
\end{proof}
The next class of numbers that we investigate is the following. 
\begin{definition}[Wu]\label{def:regular-number}
	A real number is called \emph{regular} if it can be written as a finite sum of strongly left-computable numbers.
\end{definition}
Note that this includes in particular all binary expansions of computably enumerable sets. Using the following helpful definition and propositions, we now show that every regular number is superspeedable.%
\begin{definition}
	Let $f \colon \IN \to \IN$ be a function which tends to infinity. Then the function $u_f \colon \IN \to \IN$ is defined by $u_f(n) := \left|\{k\in\IN\colon  f(k) = n\}\right|$ for all $n \in \IN$. 
\end{definition}
If $f$ is computable, then $u_f$ possesses a computable approximation from below, namely 
$u_f[0](n) := 0$ and, for all $n,t \in \IN$,
	\begin{equation*}
		u_f[t+1](n) := 
		\begin{cases}
			u_f[t](n) + 1 &\text{if $f(t) = n$,} \\
			u_f[t](n) &\text{otherwise.}
		\end{cases}
	\end{equation*}

\begin{proposition}
	A real number $x$ is regular if and only if there exists a computable function $f \colon \IN \to \IN$ with $\sum_{k=0}^{\infty} 2^{-f(k)} = x$, a so-called {\em name} of $f$, such that $u_f$ is bounded by some constant.\qed 
\end{proposition}
\begin{theorem}\label{jkerhjasdbhfvnsdabdsfasd}
	Every regular number is superspeedable; in particular this holds for every strongly left-computable number. 
\end{theorem}
\begin{proof}
	Let $x$ be a regular number. Fix some computable name $f \colon \IN \to \IN$ for $x$ and some constant $c \in \IN$ with $u_f(n) \leq c$ for all $n \in \IN$. We define two sequences of natural numbers $(a_n)_n$ and $(b_n)_n$ and a sequence of sets of natural numbers $(I_n)_n$ by
		\[a_n := \frac{n(n+1)}{2}, \quad b_n := a_n + n = a_{n+1} - 1, \quad I_n := \{a_n, \dots, b_n\},\]
	for all $n \in \IN$. It is easy to verify that the set $\{I_n \colon  n \in \IN\}$ is a partition of~$\IN$ and that we also have $\left|I_n\right| = n+1$ for all $n \in \IN$.
	
	We call a stage $s\in\IN$ a \emph{true stage} if we have $f(s) < f(t)$ for all $t > s$. Since $f$ tends to infinity, there are infinitely many true stages. For any $n, s \in \IN$, we say that the interval $I_n$ is \emph{incomplete} at stage $s$ if there exists some stage $t > s$ with $f(t) \in I_n$. Otherwise, we say that $I_n$ is \emph{complete} at $s$. Since $f$ tends to infinity, every interval is complete at some stage. Therefore, we can define the sequence $(\theta_n)_n$ with
	$
		\theta_n := \min\{t \in \IN \colon  I_n \text{ is complete at stage } t \}
	$
	for all $n \in \IN$. In order to show that $x$ is superspeedable, we consider several cases:
	\begin{itemize}
		\item If there are infinitely many $n \in \IN$ with $f(\IN) \cap I_n = \emptyset$, then 
		$x$ is superspeedable for the reasons considered in Proposition~\ref{asdbsdfjhsadnjdfgnassdfasdfgh}; more explicitly, in this case, consider such an $n$ with $a_n > \min_{k\in\IN} f(k)$ and let $s \in \IN$ be the latest true stage with $f(s) < a_n$. Then we have $\sum_{k=s+1}^{\infty} 2^{-f(k)} \leq c \cdot 2^{-b_n}$ due to $f(\IN) \cap I_n = \emptyset$, and we obtain
		\[
			\frac{2^{-f(s)}}{\sum_{k=s+1}^{\infty} 2^{-f(k)}} > \frac{2^{-a_n}}{c\cdot 2^{-b_n}} = \frac{2^{-a_n}}{c\cdot 2^{-(a_n + n)}} = \frac{2^n}{c}.
		\]
		As there exist infinitely many $n$ as above, $x$ is superspeedable.
		\item Otherwise, we can assume w.l.o.g.\ that $f(\IN) \cap I_n \neq \emptyset$ for all $n \in \IN$.
		
		Suppose that there are infinitely many $n \in \IN$ with some stage $s \in \IN$ at which $I_n$ is incomplete and $I_{n+1}$ is complete. Let $t > s$ be the earliest true stage at which $I_1,\dots,I_n$ are all complete. On the one hand, we have $f(t) < a_{n+1}$. On the other hand we have $\sum_{k=t+1}^{\infty} 2^{-f(k)} \leq c \cdot 2^{-b_{n+1}}$, since, in particular, $I_{n+1}$ is complete at stage $t$. We obtain
		\begin{equation*}
			\frac{2^{-f(s)}}{\sum_{k=s+1}^{\infty} 2^{-f(k)}}
			> \frac{2^{-a_{n+1}}}{c\cdot 2^{-b_{n+1}}} 
			= \frac{2^{-a_{n+1}}}{c\cdot 2^{-(a_{n+1}+n+1)}} 
			= \frac{2^{n+1}}{c}.
		\end{equation*}
		As there exist infinitely many $n$ as above, $x$ is superspeedable.
		
		\item Otherwise, we can assume w.l.o.g.\ that   $(\theta_n)_{n}$ is an increasing sequence. This also implies that $\theta_n$ is a true stage for every $n \in \IN$. Define the sequence $(d_n)_{n}$ by $d_n := b_n - f(\theta_n)$. Suppose that $(d_n)_{n}$ is unbounded; then for every $n \in \IN$ we have
		\begin{equation*}
			\frac{2^{-f(\theta_n)}}{\sum_{k=\theta_n+1}^{\infty} 2^{-f(k)}}
			= \frac{2^{-f(\theta_n)}}{c\cdot 2^{-b_{n}}} 
			= \frac{2^{b_n - f(\theta_n)}}{c}
			= \frac{2^{d_n}}{c}.
		\end{equation*}
		Therefore, $x$ is superspeedable.
		
		\item Otherwise, $(d_n)_n$ is bounded. Fix some natural number $D \geq 1$ with $d_n \leq D$ for all $n \in \IN$. Define the sequences $(T_n)_n$ and $(T_{n}[t])_{n,t}$ in $\{0, \dots, c\}^D$ as follows:
			\[
				T_n := \left(u_f(b_n - D + 1), \dots, u_f(b_n)\right), \quad
				T_n[t] := \left(u_f[t](b_n - D + 1), \dots, u_f[t](b_n)\right)
			\]
		Clearly, we have $\lim_{t\to\infty} T_n[t] = T_n$ for all $n \in \IN$. Fix some $D$-tuple $T \in \{0, \dots, c\}^D$ satisfying  $T_n = T$ for infinitely many $n \in \IN$. Assume w.l.o.g.\ that there are infinitely many odd numbers $n$ with this property, and recursively define the function $s \colon \IN \to \IN$ as follows:
		Let $s(0) := 0$, and define $s(m+1)$ as the smallest stage $t > s(m)$ such that there is some odd number $n \in \IN$ with $T_n[t] = T$ and $T_n[t] \neq T_n[s(m)]$. Clearly, $s$ is well-defined, computable and increasing. Finally, define the computable and increasing sequence $(x_n)_n$ by $x_n := \sum_{k=0}^{s(n)} 2^{-f(k)}$ for all $n\in\IN$; this clearly converges to~$x$. Let $n \in \IN$ be an even number with $T_{n+1} = T$. By the previous assumption that $(\theta_n)_n$ is increasing, there exists a uniquely determined number $m \in \IN$ with $s(m+1) = \theta_{n+1}$. Similarly, by the definition of $s$, we have $s(m) < \theta_n$. On the one hand, we have $x_{m+1} - x_{m} \geq 2^{-b_n}$. On the other hand, we have $x - x_{m+1} \leq c\cdot 2^{-b_{n+1}}$. So we finally obtain
		\[
			\frac{x_{m+1} - x_m}{x - x_{m+1}}
			\geq \frac{2^{-b_n}}{c \cdot 2^{-b_{n+1}}} 
			= \frac{2^{b_{n+1}-b_n}}{c} 
			= \frac{2^{a_{n+1}-a_n+1}}{c} 
			= \frac{2^{n+2}}{c}.
		\]
		As there exist infinitely many $n$ as above, $x$ is superspeedable.\qedhere
	\end{itemize}
\end{proof}

\vspace{0.3em}

\noindent As a corollary of the theorem, we obtain that the converse of Theorem~\ref{nsdfjsfjkldasjfnadsjsfdnsdadnjsad} does not hold.
\begin{corollary}\label{dsfdasjbndsnvnmasdfs}
	Not every superspeedable number is regainingly approximable.
\end{corollary}
\begin{proof}
	By Theorem~\ref{jkerhjasdbhfvnsdabdsfasd}, every strongly left-computable number is superspeedable. However, Hertling, H\"olzl, and Janicki~\cite{HHJ2023} constructed a strongly left-computable number that is not regainingly approximable. 
\end{proof}
Similarly, the converse of Theorem~\ref{jkerhjasdbhfvnsdabdsfasd} is not true either, as the following observation shows.
\begin{proposition}
	Not every superspeedable number is regular.
\end{proposition}
\begin{proof}
 Hertling, H\"olzl, and Janicki~\cite{HHJ2023} constructed a regainingly approximable number $\alpha$ such that for infinitely many $n$ it holds that $K(\alpha \ut n) > n$. By Theorem~\ref{nsdfjsfjkldasjfnadsjsfdnsdadnjsad}, this $\alpha$ is superspeedable. However, it cannot be regular, as it is easy to see that the initial segment Kolmogorov complexities of regular numbers must be logarithmic everywhere.
\end{proof}

\section{Superspeedability and additivity}

In this section we study the behaviour of superspeedability with regards to additivity. The following technical statement will be useful.		
\begin{lemma}\label{lem:splitting-lc-number}
	Let $\alpha$ and $\beta$ be left-computable numbers, and let $(c_n)_n$ be a computable increasing sequence of rational numbers converging to $\alpha + \beta$. Then there exist computable increasing sequences of rational numbers $(a_n)_n$ and $(b_n)_n$ converging to $\alpha$ and $\beta$, respectively, with $a_n + b_n = c_n$ for all $n \in \IN$.
\end{lemma}
\begin{proof}
	Given two computable increasing sequences of rational numbers $(d_n)_n$ and~${(e_n)_n}$ converging to $\alpha$ and $\beta$, respectively, fix~$N \in \IN$ with 
	\[-N < \min\{d_0, e_0\} \quad \text{and} \quad {-2N < c_0},\] 
	and define $a_{-1} := b_{-1} := -N$ and $c_{-1} := -2N$. We will inductively define an increasing function $s \colon \IN \to \IN$ as well as the desired sequences $(a_n)_n$ and $(b_n)_n$.
	
	For every~$n \in \IN$, define $s(n)$ such that the inequalities
	\[d_{s(n)} > a_{n-1},\quad e_{s(n)} > b_{n-1},\quad d_{s(n)} + e_{s(n)} \geq c_{n}\]
	are satisfied; this is obviously always possible. Assume by induction that $a_{n-1} + b_{n-1} = c_{n-1}$. Note that the quantities $d_{s(n)} - a_{n-1}$ and $e_{s(n)} - b_{n-1}$ and $c_n - c_{n-1}$ are positive rational numbers. Thus there exist natural numbers $p_a, p_b, q \geq 1$ and $p_c \geq 2$ with
	\[
	d_{s(n)} - a_{n-1} = \frac{p_a}{q},\quad 
	e_{s(n)} - b_{n-1} = \frac{p_b}{q},\quad
	c_{n} - c_{n-1} = \frac{p_c}{q}.
	\]	
	This implies
	\[
	\frac{p_a}{q} + \frac{p_b}{q} 
	= \left(d_{s(n)} + e_{s(n)} \right) - \left(a_{n-1} + b_{n-1}\right)
	\geq c_n - c_{n-1}
	= \frac{p_c}{q},
	\]
	hence $p_a + p_b \geq p_c$. 
	Let $m := \min\{p_a, p_c - 1\}$, and finally define $a_n := a_{n-1} + \frac{m}{q}$ and $b_{n} := b_{n-1} + \frac{p_c-m}{q}$. It is easy to verify that $(a_n)_n$ and $(b_n)_n$ are both computable and increasing sequences, converge to $\alpha$ and $\beta$, respectively, and satisfy $a_n + b_n = c_n$ for all~${n \in \IN}$.
\end{proof}
Using the lemma, we can establish the following result.
		\begin{theorem}
			Let $\alpha$ and $\beta$ be left-computable numbers such that $\alpha + \beta$ is superspeedable. Then at least one of $\alpha$ or $\beta$ must be superspeedable.
		\end{theorem}
		We point out that it can be shown that the converse does not hold.
		\begin{proof}
			Suppose that neither $\alpha$ nor $\beta$ are superspeedable. Let $\gamma := \alpha + \beta$, and let $(c_n)_n$ be a computable increasing sequence of rational numbers converging to $\gamma$. We show that there is some constant $\rho \in (0, 1)$ with $\frac{c_{n+1} - c_n}{\gamma - c_n} \leq \rho$ for all $n \in \IN$. Due to Lemma \ref{lem:splitting-lc-number}, there exist computable increasing sequences of rational numbers $(a_n)_n$ and $(b_n)_n$ converging to $\alpha$ and $\beta$, respectively, with $a_n + b_n = c_n$ for all $n \in \IN$. Since $\alpha$ and $\beta$ are both not superspeedable, there is some constant $\rho \in (0, 1)$ satisfying both $\frac{a_{n+1} - a_n}{\alpha - a_n} \leq \rho$ and $\frac{b_{n+1} - b_n}{\beta - b_n} \leq \rho$ for all $n \in \IN$. Considering some arbitrary $n \in \IN$, we obtain
			\begin{align*}
				\frac{c_{n+1} - c_n}{\gamma - c_n} 
				&= \frac{(a_{n+1} - a_n) + (b_{n+1} - b_n)}{(\alpha - a_n) + (\beta - b_n)} 
				= \frac{\frac{a_{n+1} - a_n}{\alpha - a_n} \cdot (\alpha - a_n) + \frac{b_{n+1} - b_n}{\beta - b_n} \cdot (\beta - b_n) }{(\alpha - a_n) + (\beta - b_n)} \\
				&\leq \frac{\max\left\lbrace \frac{a_{n+1} - a_n}{\alpha - a_n}, \frac{b_{n+1} - b_n}{\beta - b_n} \right\rbrace \cdot \left((\alpha - a_n) + (\beta - b_n)\right) }{(\alpha - a_n) + (\beta - b_n)}   
				\leq \rho.
			\end{align*}
			Thus, $\gamma$ is not superspeedable either.
		\end{proof}

		\section{Benignness versus randomness}
		
				Recall that Merkle and Titov~\cite[Theorem~10]{MT20} made the observation that  Martin-L\"of randomness is incompatible with speedability. In this section, we study the analogous question for the weaker randomness notion of Schnorr randomness.
		We begin with the easy observation that Schnorr randomness is incompatible with the rather demanding benignness notion of regaining approximability; the argument is a straight-forward modification of the result of Merkle and Titov.
\begin{proposition}
No regainingly approximable number can be Schnorr random.
\end{proposition} 
\begin{proof}
Let $x$ be regainingly approximable, witnessed by its left-approximation~$(x_n)_n$. Then $x$~is covered by the Solovay test $(S_n)_n$  where $S_n=(x_n-2^{-(n-1)},x_n+2^{-(n-1)})$ for all $n\in\IN$. This test is \textit{total}, that is, the sum of the measures of its components is computable; then, by a result of Downey and Griffiths~\cite[Theorem~2.4]{DG02}, $x$~cannot be Schnorr random.
\end{proof}
In light of this result, it is natural to ask how far down into the hierarchy of benignness notions the class of Schnorr randoms reaches. The answer is given by the following theorem.
\begin{theorem}\label{fkhejkvfjksjkbasdsgad}
There exists a superspeedable number which is Schnorr random.
\end{theorem}
We point out that the number we construct will have the additional property that it is not partial computably random. 

For an infinite set $A \subseteq \IN$ we write $p_A$ for the 
unique increasing function from $\IN$ to $\IN$ such that $p_A(\IN) = A$. Recall, for instance from Odifreddi~\cite[Section III.3]{odifreddi1992classical}, that a set $A$ is called \emph{dense simple} if it is computably enumerable and $p_{\bar A}$ dominates every computable function. While the complement of a dense simple set must be very thin by definition, 
we claim that there does exist such a set~$A$ whose binary expansion contains arbitrarily long sequences of zeros. 
To see this, recall that a set~$B$ is called {\em maximal} if it is computably enumerable and coinfinite and for any computably enumerable set $C$ with $B\subseteq C$, we must have that $C$ is cofinite or that $C\setminus B$ is finite; in other words, if $B$ only has trivial computably enumerable supersets. Such sets exist, for instance see Odifreddi~\cite[Section~III.4]{odifreddi1992classical}. If we partition the natural numbers into a sequence $(I_n)_n$ of successive intervals of growing length by letting $I_0:=\{0\}$, and $I_{n+1}:=\{\max I_{n}+1,\dots, \max I_{n}+n+2\}$ for all~$n\in \IN$ and fix some maximal set~$B$,  then it is easy to check that we obtain a set~$A$ as required by letting
\[A:=\left\{n\colon\; \exists k,\ell\;\left[ \begin{array}{c}
	\phantom{\;\vee} (n \in I_k \wedge k\in B) \;\vee \\
	(n \text{ is the $\ell$-th smallest element of } I_k \wedge \#(B\ut k)\geq \ell)
	\end{array}\right]\right\} .
\]
\begin{proof}[Proof of Theorem~\ref{fkhejkvfjksjkbasdsgad}]
Let $A$ be a set that is dense simple and whose binary expansion contains arbitrarily long sequences of zeros. Define $\Omega_A$ bitwise via 
			\[\Omega_A(n):=\begin{cases}
				\Omega(m) & \text{if } n=p_A(m),\\
				0 & \text{else};
			\end{cases}\]
for all $n\in\IN$; that is, $\Omega_A$ contains all bits of $\Omega$, but at those places that are elements of $A$; all other bits of $\Omega_A$ are zeros. 
Let $(A[t])_t$ be a computable enumeration of $A$ and for all~$n\in \IN$ write $p_A(n)[t]$ for the $n$-th smallest element in $A[t]$, if it already exists. Then it is easy to see that
			$(\Omega_A[t])_t$ defined bitwise via
			\[\Omega_A(n)[t]:=\begin{cases}
				\Omega(m)[t] & \text{if } n=p_A(m)[t]{\downarrow},\\
				0 & \text{else},
			\end{cases}\]
for all $n\in\IN$, is a left-approximation of $\Omega_A$; namely, the individual bits of $\Omega$ are approximated in a left-computable fashion, and as more and more bits appear in the computable enumeration of $A$, the positions where the bits of $\Omega$ get stored in $\Omega_A$ may move to the left.
			
We claim that $\Omega_A$ is as required by the theorem.
			
\begin{itemize}
\item That $\Omega_A$~is superspeedable follows from Proposition~\ref{asdbsdfjhsadnjdfgnassdfasdfgh}. Namely, recall that $A$~contains arbitrarily long blocks of bits that will permanently maintain value~$0$, that is, for all~$n$ in such a block and all~$t$, we have~$\Omega_A(n)[t]=0$. As a consequence, the left-approximation of $\Omega_A$ described above witnesses superspeedability at the infinitely many stages where for the last time a bit left of one of these blocks changes.
				
\item Note that the fact that $\bar{A}$ is dense immune implies that for every computable increasing function $f \colon \IN \to \IN$, there exists a number $m_f \in \IN$ with $\left|\bar{A} \ut f(n)\right| \leq \frac{n}{4}$ for all~${n \geq m_f}$.
				
\item Let $\alpha, \omega \colon \IN \to \IN$ be defined  via
	\[\begin{array}{l@{\;}c@{\;}l@{\;\;}r@{}c@{}l@{\;}c@{\;}r@{}c@{}l@{}l}
	\alpha(n) &:=& \min\{t \in \IN \colon& A[t] &\ut& (n+1) &=& A &\ut& (n+1) &\},\\
	\omega(n) &:=& \min\{t \in \IN \colon& \Omega[t] &\ut& (n/4) &=& \Omega &\ut& (n/4) &\}\\
\end{array}\]
for all $n\in\IN$. 	We claim that there exists a number $m_1 \in \IN$ such that  for every $n \geq m_1$ we have
				$\alpha(n) \leq \omega(n)$. This is clear since the Kolmogorov complexity of computably enumerable sets grows logarithmically in the length of its initial segments, while that of~$\Omega$ grows linearly.
				
		\item Let $\theta \colon \IN \to \IN$ be defined  via
		\[\theta(n) := \min\{t \in \IN \colon\; \Omega_A[t] \ut n = \Omega_A \ut n \}\]
%
		for all $n\in\IN$. 	We claim that there exists a number $m_2 \in \IN$ such that  for every $n \geq m_2$ we have
		$\omega(n) \leq \theta(n)$. To see this, fix some $n \geq m_\mathrm{id}$ and such that $\Omega_A\ut n$ contains at least one~$1$; write $t_n$ for~$\theta(n)$. 
		
		First, it is easy to see that if we let $\ell_n$ denote the maximal $\ell$ such that $(\Omega_A[t_n]\ut n)(\ell)=1$, then we must already have $A[t_n]\ut \ell_n = A\ut \ell_n$. But then, since we chose $n \geq m_\mathrm{id}$,  the sets~$A$ and~$A[t_n]$ contain exactly the same at least~$\ell_n - \nicefrac{n}{4}$ elements less than~$\ell_n$. Consequently,  we must have $\Omega[t_n] \ut (\ell_n-\nicefrac{n}{4}) = \Omega \ut (\ell_n-\nicefrac{n}{4})$. If we choose $m_2 \geq m_\mathrm{id}$ such that for all~$n \geq m_2$ we have $\ell_n \geq \nicefrac{n}2$, we are done. 
		To see that such an~$m_2$ exists, we argue as follows: For a given~$n\ge m_\mathrm{id}$,  $\Omega_A\ut n$ contains at least the first $k \ge \nicefrac{3}{4} \cdot n$ bits of $\Omega$; and thus if the last $1$ in $\Omega_A\ut n$ occurs before position~$\nicefrac{n}2$, then the last~$\nicefrac{k}{3}$ bits of $\Omega$ must be~$0$. But this can occur only for finitely many~$k$ as~$\Omega$ is Martin-L\"of random.
		
		\item
		Now we show that $\Omega_A$ is not partial computably random. We recursively define the partial martingale $d \colon {\subseteq} \SigmaS \to \IQ_{\geq 0}$ as follows: Let $d(\lambda) := 1$, and suppose that $d(\sigma)$ is defined for some $\sigma \in \SigmaS$. Let $t_{\sigma} \in \IN$ be the earliest stage with $\Omega_A[t_{\sigma}] \ut \left|\sigma\right| = \sigma$, if it exists. 
Then define $d(\sigma0)$ and $d(\sigma1)$ via
			\[d(\sigma0) := \begin{cases}
				\frac{3}{2} \cdot d(\sigma) &\text{if $\left|\sigma\right| \notin A[t_{\sigma}] $,} \\
				d(\sigma) &\text{otherwise,}
			\end{cases}
			\qquad
			d(\sigma1) := \begin{cases}
				\frac{1}{2} \cdot d(\sigma) &\text{if $\left|\sigma\right| \notin A[t_{\sigma}] $,} \\
				d(\sigma) &\text{otherwise.}
			\end{cases}\]
		It is clear that $d$ is a partial computable martingale. 
		We claim that $d$ succeeds on~$\Omega_A$. To see this, let $\ell \geq \max(m_\mathrm{id}, m_1, m_2)$ and assume that $d$~receives input $\sigma := \Omega_A \ut \ell$. Let $t_\sigma \in \IN$ be the first stage such that $\Omega_A [t_\sigma] \ut \ell = \sigma$. Then, by choice of~$\ell$ and~$\sigma$, we have $\alpha(\ell) \leq \omega(\ell) \leq \theta(\ell) =t_{\sigma}$, thus $A[t_\sigma] \ut (\ell + 1) = A \ut (\ell + 1)$. Thus,  $\ell \not\in A[t_{\sigma}]$ if and only if $\ell \not\in A$, and by construction, for each of the infinitely many $\ell \not\in A$, we have $d(\Omega_A\ut(\ell+1)) = \frac{3}{2} \cdot d(\sigma)$. On the other hand, for all other $\ell$, we have $d(\Omega_A\ut(\ell+1)) = d(\sigma)$. Thus, $\lim_{n\to\infty} d(\Omega_A\ut n) = \infty$, and $\Omega_A$ is not partial computably random.
				
				\item It remains to show that $\Omega_A$ is Schnorr random. For the sake of a contradiction, assume this is not the case; then, according to a result of Franklin and Stephan~\cite{FS2010}, there exists a computable martingale $d \colon \SigmaS \to \IQ_{\geq 0}$ and a computable increasing function $f \colon \IN \to \IN$ with $d(\Omega_A \ut f(n)) \geq 2^n$ for infinitely many $n \in \IN$. W.l.o.g.\ we can assume $d(\lambda) = 1$.
				
				Let $m:=\max(m_f, m_1)$ and define the partial computable martingale $d' \colon {\subseteq} \SigmaS \to \IQ_{\geq 0}$ recursively as follows: Let $d'(\sigma) := 1$ for each $\sigma$ with $|\sigma|\leq m$. Now suppose that $d'(\sigma)$ is defined for some~$\sigma \in \SigmaS$, write $\ell := \left|\sigma\right|$, and let $t_{\sigma} \in \IN$ be the earliest stage with $\Omega[t_{\sigma}] \ut \ell = \sigma$, if it exists. To define how $d'$ bets on the next bit, imitate the betting of $d$ on the~$p_{A}[t_\sigma](\ell+1)$-th bit of $\Omega_A[t_\sigma]$. It is clear that $d'$ is computable.
								
				We claim that $d'$ succeeds on $\Omega$.	To see this, fix any~$\ell \geq m$ and assume that $d'$~receives input~$\sigma := \Omega \ut \ell$.
				Let $t_\sigma \in \IN$ be the first stage such that $\Omega [t_\sigma] \ut \ell = \sigma$. By choice of $m$, this implies $p_{A}[t_\sigma](\ell+1) = p_A(\ell+1)$. By construction, $d'$ bets in the same way on $\Omega \ut (\ell+1)$ as $d$ does on $\Omega_A \ut p_A(\ell+1)$.
				
				Define the function $g \colon \IN \to \IN$ by $g(n) := f(n) - \left|\bar{A} \ut f(n)\right|$, and pick any of the, by assumption, infinitely many~$n$ with $d(\Omega_A \ut f(n)) \geq 2^n$.
				Recall that $\Omega_A \ut f(n)$ contains at most $\nicefrac{n}{4}$ bits which belong to $\bar{A}$ and each of them can at most double the starting capital. 
				This implies that $d'$, which omits exactly these bets but imitates all the others, must still  at least achieve capital $2^{\nicefrac34 \cdot n}$ on the initial segment $\Omega \ut g(n)$. This contradicts the well-known fact that $\Omega$ is partial computably random.
				\qedhere
			\end{itemize}
		\end{proof}	
		
		\section{Future research}\label{finalsection}
		
		We finish the article by highlighting possible future research directions:
		
		The first two open questions concern the positions marked $\scalebox{0.55}{\begin{tikzpicture}[unclear/.style = {draw=black,thick,circle, densely dashed, inner sep=0.12cm,node contents={}}]
				\node at (0,0) [unclear];
		\end{tikzpicture}}_{\,1}$ and $\scalebox{0.55}{\begin{tikzpicture}[unclear/.style = {draw=black,thick,circle, densely dashed, inner sep=0.12cm,node contents={}}]
		\node at (0,0) [unclear];
	\end{tikzpicture}}_{\,2}$ in Figure~\ref{sdfjsdfjdfmndsnmvbmnsdjkbf}; in both cases it is unknown whether any such numbers can exist. Note that proving the existence of $\scalebox{0.55}{\begin{tikzpicture}[unclear/.style = {draw=black,thick,circle, densely dashed, inner sep=0.12cm,node contents={}}]
	\node at (0,0) [unclear];
\end{tikzpicture}}_{\,1}$ would give a negative answer to the open question posed by H\"olzl and Janicki~\cite{HJ2023b} whether the left-computable numbers are covered by the union of the Martin-L\"of randoms with the speedable and the nearly computable numbers; in particular, it would provide an alternative way to obtain a counterexample to the question of Merkle and Titov discussed in the introduction. 
Concerning $\scalebox{0.55}{\begin{tikzpicture}[unclear/.style = {draw=black,thick,circle, densely dashed, inner sep=0.12cm,node contents={}}]
		\node at (0,0) [unclear];
\end{tikzpicture}}_{\,2}$, one might be tempted to believe that such a number could be constructed by an argument similar to that used to prove Theorem~\ref{fkhejkvfjksjkbasdsgad} but using a maximal computably enumerable set~$A$; such a set would still be dense simple by a result of Martin~\cite{martin1963theorem}, but it would only contain isolated zeros. However, it can be shown that an $\Omega_A$ constructed in this way would still end up being superspeedable.

			Next, besides Schnorr randomness, there are numerous other randomness notions weaker than Martin-L\"of randomness, such as computable randomness or weak $s$-randomness (see, for instance, Downey and Hirschfeldt~\cite[Definitions 13.5.6 and 13.5.8]{DH2010}). It is natural to ask which of them are compatible with which of the notions of benign approximability discussed in this article.
			
			Finally, in this field, many relevant properties of the involved objects are not computable; one might ask {\em how far}
			from computable they are. One framework in which questions of this type can be studied is the Weihrauch degrees, a tool to gauge the computational difficulty of mathematical tasks by thinking of them as black boxes that are given  {\em instances} of a problem and that have to find one of its admissible {\em solutions}. This model then allows comparing the ``computational power'' of such black boxes with each other (for more details see, for instance, the survey by Brattka, Gherardi, and Pauly~\cite{BGP21}). In the context of this article, we could for example ask for the Weihrauch degrees of the following non-computable tasks, with many variants imaginable:
			\begin{itemize}
				\item Given an approximation that witnesses speedability of some number as well as a desired
				constant, determine infinitely many stages at which the speed quotients of the given approximation
				beat the constant.				
				\item For a speedable number, given an approximation to it as well as a desired speed constant~$\rho$, determine another approximation of that number which achieves speed constant~$\rho$ in the limit superior.
				\item For an approximation witnessing regaining approximability of some number, determine the sequence of $n$'s at which the approximation ``catches up.''
			\end{itemize}	
		\goodbreak	

		\bibliography{bib2doi}

\end{document}